\documentclass[11pt]{amsart}
\usepackage{amssymb,latexsym,amsmath,epsfig,amsthm}
    \usepackage{float}

\makeatletter

\renewcommand\section{\@startsection {section}{1}{\z@}
{-30pt \@plus -1ex \@minus -.2ex}
{2.3ex \@plus.2ex}
{\normalfont\normalsize\bfseries\boldmath}}

\renewcommand\subsection{\@startsection{subsection}{2}{\z@}
{-3.25ex\@plus -1ex \@minus -.2ex}
{1.5ex \@plus .2ex}
{\normalfont\normalsize\bfseries\boldmath}}

\renewcommand{\@seccntformat}[1]{\csname the#1\endcsname. }

\makeatother

\makeatletter
\def\thm@space@setup{%
  \thm@preskip=0.5cm
  \thm@postskip=\thm@preskip 
}
\makeatother

\usepackage{yfonts}

\usepackage{babel}
\usepackage[utf8]{inputenc}
\usepackage[T1]{fontenc}
\usepackage{lmodern}
\usepackage[width=150mm,top=30mm,bottom=30mm,centering,vcentering]{geometry}
\usepackage{amsthm}
\usepackage{multicol}
\usepackage{fancyhdr}
\usepackage{lastpage}
\usepackage{stmaryrd}
\usepackage{mathtools}
\usepackage{centernot}
\usepackage{etoolbox}
\usepackage{enumitem}

\usepackage{tabularx}
\usepackage{setspace}

\usepackage{subfig}

\makeatletter
\newcommand{\subalign}[1]{%
  \vcenter{%
    \Let@ \restore@math@cr \default@tag
    \baselineskip\fontdimen10 \scriptfont\tw@
    \advance\baselineskip\fontdimen12 \scriptfont\tw@
    \lineskip\thr@@\fontdimen8 \scriptfont\thr@@
    \lineskiplimit\lineskip
    \ialign{\hfil$\m@th\scriptstyle##$&$\m@th\scriptstyle{}##$\hfil\crcr
      #1\crcr
    }%
  }%
}
\makeatother

\DeclareMathOperator{\ord}{ord}

\newtheorem{theorem}{Theorem}
\newtheorem{lemma}[theorem]{Lemma}

\newtheorem{corollary}[theorem]{Corollary}
\newtheorem{remark}[theorem]{Remark}

\usepackage[colorlinks,linkcolor=black,urlcolor=blue,citecolor=black]{hyperref} 
\usepackage{cleveref}

\usepackage{mathtools}

\newcommand*{\Z}{\mathbb{Z}}

\newcommand*{\Q}{\mathbb{Q}}

\newcommand\restr[2]{{
  \left.\kern-\nulldelimiterspace 
  #1 
  \littletaller
  \right|_{#2} 
  }}

\newcommand{\littletaller}{\mathchoice{\vphantom{\big|}}{}{}{}}

\newcommand{\Gal}{\mathrm{Gal}}

\setlist[enumerate]{itemsep=0.3em}

\title{Explicit Prime Densities for Lucas Sequence Rank Divisibility}
\author{Joaquim Cera Da Concei\c c\~ao}
\address{Normandie Universit\'e, UNICAEN, CNRS, LMNO, 14000 Caen, France}
\email{\tt joaquim.cera-daconceicao@unicaen.fr}
\urladdr{\href{https://jceradaconceicao.github.io}{\tt https://jceradaconceicao.github.io}}


\begin{document}

\begin{abstract}
    \leftskip=0.3cm
    \rightskip=0.3cm
    Let $U$ be a Lucas sequence, $p$ be a prime, and $\rho_U(p)$ be the rank of appearance of $p$ in $U$, that is, the least positive integer $n$ such that $p\mid U_n$, if it exists. We derive closed-form formulas for the Dirichlet density of primes $p$ for which $d\mid \rho_U(p)$, where $d\geq 1$ is a fixed integer. Our results complete the work of Sanna ($2022$) by covering all $U$ and all $d\geq 1$.
\end{abstract}

\maketitle
\setcounter{tocdepth}{1}
\tableofcontents

\onehalfspacing

\section{Introduction}

A {\it Lucas sequence} $U=U(a_1,a_2) = (U_n)_{n\geq 0}$ with non-zero parameters $a_1,a_2 \in \Z$ is defined by $U_0=0$, $U_1=1$, and $U_{n+2}=a_1U_{n+1}-a_2U_n$ for all $n\geq 0$. Given a prime number $p$, we denote by $\rho_U(p)$ the smallest positive $n$ at which $p\mid U_n$. This integer is called the {\it rank of appearance} of $p$ in $U$ and exists for all $p\nmid a_2$. In this paper, we study the Dirichlet density of the set $\mathcal{R}_U(d)$ of primes $p\nmid a_2$ such that $\rho_U(p)$ is divisible by a fixed integer $d\geq 1$. This was first studied by Hasse \cite{Hasse65,Hasse66} for sequences $U(a+1,a)$ with $a\in\Z$ square-free, $|a|\geq 2$, and $d$ a prime number. He settled the existence and found an explicit formula for the Dirichlet density. For instance, Hasse found a density of $17/24$ in the case $|a|=d=2$, and $d^2/(d^2-1)$ otherwise. Following his work, many authors considered the problem of determining the density for other Lucas sequences and values of $d$.

Let $f_U(X) = X^2-a_1X+a_2$ be the characteristic polynomial of $U$. The problem was completely solved by Wiertelak \cite{WiertelakI,WiertelakII,WiertelakIV} when $f_U(X)$ is reducible. In 1985, Lagarias \cite{Lagarias85} computed the Dirichlet density of $\mathcal{R}_U(2)$ for three Lucas sequences with irreducible polynomial. In particular, he proved that $\mathcal{R}_F(2)$ has density $2/3$, where $F=U(1,-1)$ is the Fibonacci sequence. Later on, the Fibonacci case was completed for all $d\geq 1$ by Cubre and Rouse \cite{CubreRouse14} through the study of algebraic groups. More recently, Sanna \cite{Sa2022} made an important contribution to this problem by computing the Dirichlet density for all odd $d\geq 1$ with $3\nmid d$ if the splitting field of $f_U$ is $\Q(\sqrt{-3})$. This was done using a method of Moree \cite{Mor2005}. Prior to Sanna's result, Ballot \cite{Ballot13} considered the case of the splitting field being $\Q(i)$ or $\Q(\sqrt{-3})$, and $d\in\{2,4\}$ or $d\in\{3,6\}$ respectively. We want to stress that a cyclotomic splitting field causes many problems when $d$ is divisible by one of $2$ or $3$.

Let $a,b\in\Bar{\Q}$ be the roots of $f_U$ and $K=\Q(a)$. We denote by $\Delta=a_1^2-4a_2$ the discriminant of $f_U$ and $\gamma = a/b$ its root quotient. Then, the set $\mathcal{R}_U(d)$ is equal, up to finitely many exceptions, to
\[
\mathcal{R}_\gamma(d)=\{p : p\nmid 2a_2\Delta \text{ and } d\mid \mathrm{ord}_\pi(\gamma) \text{ for all $\pi\in \mathcal{O}_K$ and $p=\pi\cap \Z$}\},
\]
where $\mathcal{O}_K$ is the ring of integers of $K$ and $\ord_\pi(\gamma)$ is the order of $\gamma \bmod \pi$ in the multiplicative group $(\mathcal{O}_K/\pi)^\times$. Indeed, we have $\rho_U(p) = \mathrm{ord}_\pi(\gamma)$, by \cite[Lemma 4.1]{Sa2022}. The goal of our paper is to find the Dirichlet density of $\mathcal{R}_U(d)$, if it exists, in all the cases left by Sanna's theorem. Moreover, we want to write it in a closed-form formula, meaning that only finitely many operations are required to compute its value. By the above, the sets $\mathcal{R}_U(d)$ and $\mathcal{R}_\gamma(d)$ have the same Dirichlet density, when it exists. Therefore, we study the set $\mathcal{R}_\gamma(d)$, which is more suitable for some of our calculations.

In the statement of \cite[Theorem 1.1]{Sa2022}, Sanna assumed that $d$ is odd and not divisible by $3$ when $\Delta_K=-3$. However, these restrictions are not used in the proof of the existence of the density, nor in the proof of the error term in the asymptotic formula. Indeed, \cite[Lemma 5.1]{Sa2022} is stated without them and, while they appear in the statement of \cite[Lemma 5.3]{Sa2022}, the proof does not invoke them. The proof of the main theorem only relies on these lemmas and \cite[Lemmas 5.2 and 5.4]{Sa2022}. It is in the latter that the assumption on $d$ is required to compute a closed-form formula of the density. Let $x>1$ and denote by $\mathcal{R}_\gamma(d;x)$ the number of primes $p\in \mathcal{R}_\gamma(d)$ with $p\leq x$. Let $K_{n,d} = K(\zeta_n,\gamma^{1/d})$, where $d\mid n$ are integers and $\zeta_n$ is a primitive $n$-th root of unity. We write $\mathrm{Li}$, $\omega$, and $\varphi$ for the logarithmic integral, the number of distinct prime divisors, and the Euler totient functions respectively. With a slight change of the proof of \cite[Lemma 5.2]{Sa2022}, the result \cite[Theorem 1.1]{Sa2022} can be restated, with our notation, as follows.

\begin{theorem}\label{theorem: existence of the density}
    Let $d$ be an integer. There exists an absolute constant $B>0$, such that for every $x>\exp(Bd^{40})$, we have
    \[
    \mathcal{R}_\gamma(d;x) = \delta_\gamma(d) \mathrm{Li}(x) + \mathcal{O}_\gamma\left( \frac{d}{\varphi(d)} \cdot \frac{x(\log\log x)^{\omega(d)}}{(\log x)^{9/8}} \right),
    \]
    where
    \begin{equation}\label{equation: definition of delta_gamma(d)}
        \delta_\gamma(d) = \sum_{v\mid d^\infty} \sum_{u\mid d} \frac{\mu(u)(1+[\sigma_{u,v} \text{ exists}])}{[K_{dv,uv}: \Q]},
    \end{equation}
    and $[\sigma_{u,v} \text{ exists}]$ is a Boolean that checks the existence of a $\sigma_{u,v}$ in  $Gal(K_{dv,uv}/\Q)$ satisfying $\sigma_{u,v}(\sqrt{\Delta}) = -\sqrt{\Delta}$, $\sigma_{u,v}(\zeta_{dv}) = \zeta_{dv}^{-1}$, and $\sigma_{u,v}(\gamma^{1/uv}) = \gamma^{-1/uv}$.
\end{theorem}

In this paper, we compute closed-form formulas for the series defined in \eqref{equation: definition of delta_gamma(d)} in all the cases left by Sanna \cite{Sa2022}. This includes the more difficult cases of a cyclotomic splitting field $K$ discussed above. We note that a particular case of our results was obtained independently by Luo, Hong, and Liu \cite{LuoHongLiu26}. They extended Sanna's result to even integers $d\geq 2$ under the assumption that $K$ is not cyclotomic and $\gamma$ is not an $n$-th power in $K$ for all $n\geq 2$. They used a method similar to Cubre and Rouse \cite{CubreRouse14}. To simplify our calculations, we separate $\delta_\gamma(d)$ into the smaller values
\begin{equation}\label{equation: definition of delta+ and delta-}
    \delta_\gamma^+(d) := \sum_{v\mid d^\infty} \sum_{u\mid d} \frac{\mu(u)}{[K_{dv,uv}:\Q]} \quad \text{and} \quad \delta_\gamma^-(d) := \sum_{v\mid d^\infty} \sum_{u\mid d} \frac{\mu(u)[\sigma_{u,v} \text{ exists}]}{[K_{dv,uv}:\Q]},
\end{equation}
according to the sum $1+[\sigma_{u,v} \text{ exists}]$. The number $\delta_\gamma^+(d)$ corresponds to the Dirichlet density of $\mathcal{R}_\gamma^+(d)$\label{nom:R+ et R- sur Z}, the set of primes $p\nmid a_2$ whose rank is divisible by $d$ and with Legendre symbol $(\Delta/p) = 1$. For $\delta_\gamma^-(d)$, we have $(\Delta/p)=-1$ and we call $\mathcal{R}_\gamma^-(d)$ the corresponding set.

In Section 2, we introduce the notation and basic definitions that will be used throughout this paper.

The third section is divided into two parts. First, we give results on powers in cyclotomic extensions of a quadratic field $K$. For instance, we explicitly construct a $4$-th root of a norm $1$ element $\gamma\in \Q(i)$ in Lemma \ref{equation: fourth root of gamma}. Next, we prove several results on {\it cyclotomic-Kummer extensions} of $K$. They are extensions of the form $K(\zeta_n, \gamma^{1/d})$, where $n$ and $d\mid n$ are positive integers, and $\zeta_n$ is a primitive $n$-th root of unity. We first compute the degree over $\Q$ of such extensions, and then give necessary and sufficient conditions for special automorphisms to exist in their Galois groups. 

Let $\zeta\in K$ be a root of unity and $h(\zeta)$ be the largest $n\geq 1$ such that $\zeta\gamma\in (K^\times)^n$. In Section \ref{section: h=h(1)}, we compute closed-form formulas for the density defined in \eqref{equation: definition of delta_gamma(d)} in the case where $h(1)\geq h(\zeta)$ for all roots of unity $\zeta\in K$.

In Section \ref{section: remaining cases}, we turn to the cases where $h(1)$ is not the maximum over the set of integers $h(\zeta)$ for roots of unity $\zeta\in K$. In these cases, computing $\delta_\gamma(d)$ can be reduced to the computation of $\delta_{\gamma'}(d)$ for some $\gamma'$ for which $h(1)$ is maximal. In particular, if $h(\zeta)$ is the maximum, then $\gamma' = \zeta \gamma$. Consequently, we find that $\delta_\gamma(d)$ has a closed-form formula for all $d\geq 1$ using Section \ref{section: h=h(1)}. Thus, these two sections provide a complete solution to the problem of computing $\delta_\gamma(d)$.

Finally, we provide numerical demonstrations of our results in Section \ref{section: numerical data}. For several choices of $K$, $\gamma$ and $d$, we display the value of $\delta_\gamma(d)$ next to its experimental value computed via SageMath \cite{sagemath} programs.

\section{Notation}

Throughout this paper, we let $\gamma =a/b$ denote the quotient of the roots of the polynomial $f_U(X):=X^2-a_1X+a_2$. In addition, we assume that $\gamma$ is not a root of unity, as otherwise one has $U_n = 0$ for infinitely many $n\geq 1$ and the problem becomes trivial. We may also assume that $f_U$ is irreducible since the reducible case was settled by Wiertelak \cite{WiertelakIV}. We let $K=\Q(a)$, $\mu(K)$ be the set of roots of unity contained in $K$ and $\mu_n(K)$ be the set of $n$-th root of unity in $K$. Given $\zeta\in \mu(K)$, we let $h(\zeta)$ be the largest integer $n\geq 1$ such that $\zeta\gamma \in (K^\times)^n$. A central object in our paper is
\[
h := \max_{\zeta\in\mu(K)} h(\zeta).
\]
We use the letters $p$ and $d,n$ for a prime number and positive integers respectively. The $p$-adic valuation will be denoted by $v_p$. Given a field $L$, we denote by $\Delta_L$ its absolute discriminant and by $(L^\times)^n$ the subgroup of $n$-th powers of $L$. We let $d^\infty$ be the {\it supernatural number} 
\[
d^\infty = \prod_{p\mid d} p^\infty,
\]
where we allow prime numbers to have infinite $p$-adic valuation. To check the validity of a proposition $\mathcal{P}$, we use the Iverson brackets $[\mathcal{P}]$ , which is a Boolean function that equals $1$ if and only if $\mathcal{P}$ is true. We use $\zeta_n$ to denote a primitive $n$-th root of unity. We denote by $(m,n)$ and $[m,n]$ the $\gcd$ and the $\mathrm{lcm}$ of integers $m$ and $n$ respectively. The letters $\varphi$ and $\mu$ stand for the Euler totient function and the M\"obius function respectively. We denote by $K_{n,d}$ the cyclotomic-Kummer extension $K(\zeta_n,\gamma^{1/d})$ of $K$ when $d\mid n$. Finally, we write $h_m = (h,m^\infty)$ for all $m\geq 1$.

\section{Preliminaries on Cyclotomic-Kummer Extensions}

This section is divided into two subsections. In the first one, we study whether $\gamma$ can be expressed as a power in the cyclotomic fields $K(\zeta_n)$. In the second, we compute the degree of cyclotomic-Kummer extensions $K_{n,d}=K(\zeta_n,\gamma^{1/d})$ of $K$, where $d\mid n$ is a positive integer. In the case $\Delta>0$, we give necessary and sufficient conditions for the existence of certain automorphisms in their Galois group. When $\Delta<0$, we show that these automorphisms always exist. Below, we display two theorems found in literature that we use on many occasions.

\begin{theorem}\label{theorem: irreducibility of X^n-a}
    Let $K$ be a field and $a\in K^\times$. Then, $X^n-a$ is irreducible over $K$ if and only if $a\not \in (K^\times)^p$ for all primes $p\mid n$ and $a\not\in -4(K^\times)^4$ if $4\mid n$.
\end{theorem}

\begin{proof}
    See \cite[Chapter 8, Theorem 1.6]{Kar1989}.
\end{proof}

\begin{theorem}\label{theorem: cyclotomic-Kummer abelian iff}
    Let $K$ be a number field and $m= \#\mu_n(K)$. Then, for all $a\in K$, the extension $K(\zeta_n,a^{1/n})/F$ is abelian if and only if $a^m \in (K^\times)^n$.
\end{theorem}

\begin{proof}
    See \cite[Chapter 8, Theorem 3.2]{Kar1989}.
\end{proof}

\subsection{Powers in cyclotomic extensions}

Throughout this subsection, we consider an arbitrary $\gamma\in K$. We determine at which condition a non-square $\gamma$ becomes a square in $K(\zeta_n)$, where $n\geq 1$. We also study the fourth roots of $\gamma$ when $\Delta_K=-4$, and its cube roots when $\Delta_K=-3$.

\begin{lemma}\label{lemma: square root in cyclotomic extensions}
    Let $\gamma \in K\setminus(\Q \cup K^\times)^2$ with $|N_{L/K}(\gamma)|=1$ and write $\gamma=u+v\sqrt{\Delta_K}$, where $u,v\in \Q^\times$. Letting $c=(u-1)/2$, define
    \[
    K_1 = \Q(\sqrt{c}) \quad \text{and} \quad K_2 = \Q(\sqrt{c/\Delta_K})
    \]
    and let $\Delta_1$ and $\Delta_2$ be their respective absolute discriminants. Then, $\sqrt{\gamma} \in K(\zeta_n)$ if and only if $N_{K/\Q}(\gamma)=1$, and $\Delta_1\mid n$ or $\Delta_2\mid n$.
\end{lemma}

\begin{proof}
    First, assume that $K$ is one of $\Q(i)$ or $\Q(\zeta_3)$. In that case, the norm can only be equal to $1$. The polynomial $X^4-2uX^2+1$ annihilates $\sqrt{\gamma}$ and is irreducible over $\Q$ because $\gamma \not \in (K^\times)^2$. It has Galois group $C_2\times C_2$ by \cite[Lemma 4.2]{Sa2024}. The proof follows as \cite[Lemma 4.6]{Sa2024}, which also deals with the non-cyclotomic cases, and we use \cite[Lemma 3]{Schinzel70} for the condition on the discriminants.
\end{proof}

\begin{lemma}\label{lemma: fourth root in cyclotomic fields}
    Assume that $\Delta_K=-4$. Let $\gamma \in K\setminus (\Q \cup (K^\times)^2)$ with $N_{K/\Q}(\gamma)=1$ and write $\gamma = u+v\sqrt{-4}$, where $u,v\in \Q^\times$. Then, a fourth root of $\gamma$ is given by
    \begin{equation}\label{equation: fourth root of gamma}
        \gamma^{1/4} = \bigg(\frac{1+\sqrt{c}}{2}\bigg)^{1/2} + \frac{v\sqrt{-4}}{2|v|} \bigg(\frac{1-\sqrt{c}}{2}\bigg)^{1/2},
    \end{equation}
    where $c=(u-1)/2$. Moreover, the field $L=K(\gamma^{1/4})$ is abelian over $\Q$ and its conductor has the form $\mathfrak{f}(L) = 2^{a}p_1\cdots p_s$, where $a\in \{2,3,4\}$ and the $p_i$'s are the odd primes that ramify in $L$.
\end{lemma}

\begin{proof}
    Let $z$ be the right-hand side of \eqref{equation: fourth root of gamma}. We have
    \[
    z^2 =  \sqrt{c} + \frac{v\sqrt{-4}}{2|v|} \cdot \sqrt{1-c}  = \sqrt{c} + \frac{v\sqrt{-4}}{2\sqrt{c}},
    \]
    where we used that $c+v^2/c = 1$, because $z$ has modulus $1$. We find $\gamma$ by squaring a second time and using $c-v^2/c = u$. Next, write $L= \Q(i,\gamma^{1/4})$. One can show that a field automorphism of $L/\Q$ is entirely determined by the relations
    \[
    \sigma_{j,k} \,\colon 
        \begin{cases}
            i & \! \mapsto \; ij,\\
            \gamma^{1/4} & \! \mapsto \; i^k \gamma^{j/4},
        \end{cases}
    \]
    where $j\in\{-1,1\}$ and $0\leq k\leq 3$. We used $N_{K/\Q}(\gamma)=1$, that is, $\Bar{\gamma} =\gamma^{-1}$, to obtain the relations. We can see that they commute, and thus $L/\Q$ is abelian. Finally, we may write $L = \Q(i,\alpha)$, where $2\alpha^2 = 1+\sqrt{c}$. Therefore, $L$ is the compositum of $K$ and $F=\Q(\alpha)$. By \cite[Chapter 2, Proposition 4.1.1]{Gras2003}, we have
    \[
    \mathfrak{f}(L) = [\mathfrak{f}(K), \mathfrak{f}(F) ] = [4,\mathfrak{f}(F)].
    \]
    Next, the minimal polynomial of $\alpha$ is $X^4-X^2+(1-c)/4$. By \cite[Lemma 4.2]{Sa2024}, we see that $F/\Q$ has Galois group $C_4$. Hence $\mathfrak{f}(F) = 2^bp_1\cdots p_s$, where $b\in\{0,2,3\}$, by the formula proved by Spearman and Williams \cite{SpearmanWilliams1997}.
\end{proof}

\begin{lemma}\label{lemma: cube root in Q(zeta_3,zeta_n)}
    Assume that $\Delta_K = -3$. Let $\gamma\in K\setminus (\Q \cup (K^\times)^3)$ with $N_{K/\Q}(\gamma)=1$ and write $\gamma = u+v\sqrt{-3}$, where $u,v\in \Q^\times$. Then, a cube root of $\gamma$ is given by
    \begin{equation}\label{equation: cube root of gamma}
        \gamma^{1/3} = \frac{r}{2} + \frac{v}{r^2-1}  \sqrt{-3},
    \end{equation}
    where $r$ is a root of $R(X) =X^3-3X-2u$, which is irreducible over $\Q$, has three real roots, and has abelian splitting field $L$. Moreover, if $p_1<\cdots <p_s$ denote the ramified primes in $L$ outside of $3$, then the conductor of $L$ is of the form
    \[
    \mathfrak{f}(L ) = 3^a p_1\cdots p_s,
    \]
    where $a\in\{0,2\}$, and $\gamma^{1/3}\in K(\zeta_n)$ if and only if $\mathfrak{f}(L) \mid n$.
    
\end{lemma}

\begin{proof}
    The formula for $\gamma^{1/3}$ was obtained by Cavallo \cite{Cavallo24} for real quadratic fields. It turns out that the same formula holds in our case. Let $z$ denote the right-hand side of \eqref{equation: cube root of gamma}. Note that $x+y\sqrt{-3}$, where $x,y\in \Q$, is a cube root of $\gamma$ if and only if
    \[
    x(x^2-9y^2) = u \quad \text{and} \quad 3y(x^2-y^2)=v.
    \]
    Here, $x=r/2$ and $y=v/(r^2-1)$. Let us first compute $y^2$. Since $\gamma$ has norm $1$, we have $u^2+3v^2=1$ and
    \[
    y^2 = \frac{v^2}{(r^2-1)^2} =\frac{1-u^2}{3(r^2-1)^2} =\frac{4-( r^3-3r)^2}{12(r^2-1)^2},
    \]
    where we used $2u = r^3-3r$. Then, expanding the numerator, we obtain
    \[
    y^2 = \frac{-r^6+6r^4-9r^2+4}{12(r^2-1)^2} =  \frac{-(r^2-1)^2(r^2-4)}{12(r^2-1)^2} = \frac{4-r^2}{12}.
    \]
    Replacing the value of $y^2$ in both $x(x^2-9y^2)$ and $3y(x^2-y^2)$ yields $u$ and $v$ respectively. Hence $z^3=\gamma$. Next, the discriminant $\mathrm{Disc}(R)=4\cdot 27\cdot(1-u^2) = (18v)^2$ of $R(X)$ is a positive square. Thus, $R$ has three real roots and $\Gal(L/\Q) \cong C_3$. Finally, we have $K(\gamma^{1/3}) = K(r)$ by \eqref{equation: cube root of gamma}. Therefore, we have $\gamma^{1/3}\in K(\zeta_n)$ if and only if $K(r)\subset K(\zeta_n)$. 
    This is equivalent to $L=\Q(r)$ being a subfield of $ \Q(\zeta_n)$ since $\Q(r)$ is totally real. By definition of the conductor, we obtain $\gamma^{1/3}\in K(\zeta_n)$ if and only if $\mathfrak{f}(L) \mid n$. The formula for the conductor of a cubic field was computed by Hasse \cite{Hasse30}. (See also the work of Huard, Spearman, and Williams \cite{HSW94} for another characterization.)
\end{proof}

\subsection{Results on cyclotomic-Kummer extensions}

We assume that $h=h(1)$ and write $\gamma = \gamma_0^h$ for some $\gamma_0\in K$. We study the cyclotomic-Kummer extensions $K_{n,d}=K(\zeta_n,\gamma^{1/d})$, where $d\mid n$ is a positive integer. In the next theorem, we compute the minimal polynomial of $\gamma^{1/d}$ over $K(\zeta_n)$, which is essential for the determination of the degree of $K_{n,d}/\Q$. We write $d_0 = d/(d,h)$ and $h_0 = h/(d,h)$ throughout this subsection.

\begin{theorem}\label{theorem: minimal polynomial when h=h(1)}
    The minimal polynomial of $\gamma^{1/d}$ over $K(\zeta_n)$ is $X^{d_0/t} -\gamma_0^{h_0/t}$, where
    \[
    t := \max( m\mid \#\mu(K) : mh_m \mid d \text{ and } \gamma^{1/mh_m} \in K(\zeta_n) ).
    \]
\end{theorem}

\begin{proof}
    We use a slightly different definition of $t$ for the proof. We observe that $mh_m\mid d$ if and only if $m \mid d_0$, and $\gamma^{1/mh_m} \in K(\zeta_n)$ if and only if $\gamma_0^{h_0/m} \in K(\zeta_n)$. In each case, the polynomial $X^{d_0/t}-\gamma_0^{h_0/t}$ annihilates $\gamma^{1/d}$ and we use Theorem \ref{theorem: irreducibility of X^n-a} to prove its irreducibility. Let $p$ be a prime divisor of $d_0/t$. 

    Assume that $v_p(t) < v_p(\#\mu(K))$. Since $p\mid d_0$, we have $ph_p \mid d$. Then, by definition of $t$ and its maximality, we know that $\gamma_0^{h_0/t}$ is not a $p$-th power in $K(\zeta_n)$.

    Assume that $v_p(t) = v_p(\#\mu(K))$. By contradiction, assume $\gamma_0^{h_0/t} \in (K(\zeta_n)^\times)^p$, or equivalently, assume $\gamma_0^{h_0} \in K\cap (K(\zeta_n)^\times)^{p^{1+v_p(t)}}$. This is equivalent to
    \[
    \gamma_0^{h_0 p^{v_p(t)}} = x^{p^{1+v_p(t)}},
    \]
    for some $x\in K$, by Theorem \ref{theorem: cyclotomic-Kummer abelian iff}, because $p^{v_p(t)}$ is the number of $p$-th roots of unity in $K$. It follows that $\gamma_0^{h_0} = \zeta_{p^{v_p(t)}}^k x^{p}$ for some $k\in \Z$. However, this is a contradiction to the maximality of $h$ if $k\equiv 0 \pmod{p^{v_p(t)}}$, and to $h=h(1)$ otherwise.

    Finally, assume that $p=2$ and $4$ divides $ d_0/t$. In particular, we have $4\mid n$. Thus, if we write $\gamma_0^{h_0} = -4x^4$ for some $x\in K(\zeta_n)$, then $ \gamma_0^{h_0} = (2ix^2)^2$ is a square. This is a contradiction to the above.
\end{proof}

\begin{corollary}\label{corollary: degree of Kummer extensions}
    Let $t$ be defined as in Theorem \ref{theorem: minimal polynomial when h=h(1)}. Then,
    \[
    [K_{n,d}:\Q] = \frac{d\varphi(n)}{(d,h) t} \cdot 
    \begin{cases}
        2, & \text{if $\Delta_K \nmid n$}; \\
        1, & \text{otherwise}.
    \end{cases}
    \]
\end{corollary}

\begin{proof}
    We know that $K \subset \Q(\zeta_n)$ if and only if $\Delta_K\mid n$ by \cite[Lemma 3]{Schinzel70}. Therefore, we have $[K(\zeta_n):\Q] = \varphi(n)\cdot 2^{[\Delta_K\nmid  n]}$ and the result follows by Theorem \ref{theorem: minimal polynomial when h=h(1)}.
\end{proof}

In future sections, we use Lemmas \ref{lemma: square root in cyclotomic extensions}, \ref{lemma: fourth root in cyclotomic fields}, and \ref{lemma: cube root in Q(zeta_3,zeta_n)} to explicitly determine $t$, defined in Theorem \ref{theorem: minimal polynomial when h=h(1)}. For instance, if $\Delta_K\not \in \{-4,-3\}$, then $\mu(K) = \{\pm 1\}$. By Lemma \ref{lemma: square root in cyclotomic extensions}, we have $\gamma^{1/2h_2}\in K(\zeta_n)$ if and only if $N_{K/\Q}(\gamma^{1/h_2})=1$ and one of $\Delta_1$ or $\Delta_2$ divides $n$, where the $\Delta_i$'s are defined in the lemma. Thus, $t = 2$ if $2h_2\mid d$, $N_{K/\Q}(\gamma^{1/h_2})=1$ and one of $\Delta_1$ or $\Delta_2$ divides $n$, and $t=1$ otherwise.

\begin{theorem}\label{theorem: existence of sigma_u,v}
    Assume $\Delta_K>0$. Let $\sigma: K_{n,d} \to K_{n,d}$ be such that $\sigma|_{\Q} = \mathrm{id}$ and
    \[
    \sigma \colon 
    \begin{cases}
        \sqrt{\Delta}  &\!\longmapsto \, -\sqrt{\Delta}; \\
        \zeta_n &\!\longmapsto \, \zeta_n^{-1};\\
        \gamma^{1/d} &\!\longmapsto \, \gamma^{-1/d}.
    \end{cases}
    \]
    If $2h_2\nmid d$ or $\gamma^{1/h_2} \not \in (K(\zeta_n)^\times)^2$, then $\sigma$ belongs to $ \Gal(K_{n,d}/\Q)$ if and only if the two following conditions are satisfied:
    \begin{enumerate}
        \item $\Delta_K\nmid n$;

        \item $h_2\nmid d$, or $h_2\mid d$ and $N_{K/\Q}(\gamma^{1/h_2}) = 1$.
    \end{enumerate}
    If $2h_2 \mid d$ and $\gamma^{1/h_2}\in (K(\zeta_n)^\times)^2$, then $\sigma$ belongs to $\Gal(K_{n,d}/\Q)$ if and only if the two following conditions are satisfied:
    \begin{enumerate}
        \item $\Delta_K\nmid n$;

        \item $\sigma_0(\gamma^{1/2h_2}) = \gamma^{-1/2h_2}$,
    \end{enumerate}
    where $\sigma_0 \in \Gal(K(\zeta_n)/\Q)$ satisfies $\sigma_0(\sqrt{\Delta}) = -\sqrt{\Delta}$ and $\sigma_0(\zeta_n)=\zeta_n^{-1}$.
\end{theorem}

\begin{proof}
    From the proof of \cite[Lemma 4.2]{Sa2022}, we know $\sigma_0$ exists if and only if $\Delta_L \nmid n$. Since $\sigma|_{K(\zeta_n)} = \sigma_0$, it suffices to find necessary and sufficient conditions for $\sigma_0$ to be extended into $\sigma$. By Theorem \ref{theorem: minimal polynomial when h=h(1)}, let $g(X) = X^{d_0/t}-\gamma_0^{h_0/t}$ be the minimal polynomial of $\gamma^{1/d}$ over $K(\zeta_n)$. Since $K_{n,d} \cong K(\zeta_n)[X] / (g(X))$, we can extend $\sigma_0$ by sending a root of $g$ to any root of $\sigma_0 g$. Equivalently, we need
    \begin{equation}\label{equation: sigma_0 need for existence annihilation}
        \sigma_0(\gamma_0^{h_0/t}) = \gamma_0^{-h_0/t}.
    \end{equation}
    If $t=2$, then we show that this is equivalent to $\sigma_0(\gamma^{1/2h_2}) = \gamma^{-1/2h_2}$. The direct implication is trivial by raising \eqref{equation: sigma_0 need for existence annihilation} to the $(d,h')$-th power, where $h'=h/h_2$. For the converse, taking the $(d,h')$-th root on both sides, we obtain
    \[
    \sigma_0(\gamma^{1/2h_2(d,h')}) =  \sigma_0(\gamma_0^{h_0/2}) = \zeta_{(d,h')}^k\gamma_0^{-h_0/2},
    \]
    for some $k\in \Z$. Squaring both sides, we have $\sigma_K(\gamma_0^{h_0}) = \zeta_{(d,h')}^{2k}\gamma_0^{-h_0}$, which holds in $K$. Hence $\zeta_{(d,h')}^{2k} = 1$ and, because $2\nmid (d,h')$, we have $\zeta_{(d,h')}^k=1$ as well.

    If $t=1$, then we have two cases. If $h_2\mid d$, then $\gamma_0^{h_0}$ is a square in $K$. Moreover, because $\sigma_K(\gamma)=\gamma^{-1}$, we have $\sigma_K(\gamma_0^{h_0/2}) =\pm\gamma_0^{-h_0/2}$ and it follows that $\sigma_0(\gamma_0^{h_0}) = \gamma_0^{-h_0}$ by squaring both sides. If $h_2\nmid d$, then $\sigma_0(\gamma_0^{h_0}) = \gamma_0^{-h_0}$ if and only if $\sigma_K(\gamma^{1/h_2}) = \gamma^{-1/h_2}$. This is equivalent to the norm of $\gamma^{1/h_2}$ being $1$.
\end{proof}

We end this section with a short lemma on the existence of the automorphism $\sigma_{u,v}$ defined in Theorem \ref{theorem: existence of the density}. This will ease calculations when $\Delta<0$, particularly when $K$ is a cyclotomic field. Indeed, a direct consequence is that $\delta_\gamma(d) = 2\delta_\gamma^+(d)$.

\begin{lemma}\label{lemma: existence of sigma_u,v when D<0}
    Assume that $\Delta_K<0$. Then, the Galois group of $K_{dv,uv}/\Q$ contains an automorphism $\sigma_{u,v}$ as defined in Theorem \ref{theorem: existence of the density} for all $u\mid d$ and $v\mid d^\infty$.
\end{lemma}

\begin{proof}
    Let $\tau$ be the complex conjugation. Because $K_{dv,uv}/\Q$ is a finite Galois extension, the restriction $\tau|_{K_{dv,uv}}$ is a field automorphism of $K_{dv,uv}$. Since $\gamma\in K$ and $\Delta_K<0$, we have $\tau(\gamma)=\gamma^{-1}$. Hence $\tau(\gamma^{1/uv}) = \gamma^{-1/uv}$ and $\sigma_{u,v}=\tau|_{K_{dv,uv}}$.
\end{proof}

\section{The case \texorpdfstring{$h=h(1)$}{TEXT}} \label{section: h=h(1)}

In this section, we derive closed-form formulas for $\delta_\gamma(d)$ in the case $h=h(1)$. To do so, we compute $\delta_\gamma^+(d)$ and $\delta_\gamma^-(d)$ separately. Recall that $\gamma = \gamma_0^h$ for some $\gamma_0\in K$ and that $h_m =(h,m^\infty)$ for all $m\geq 2$. We define the Boolean variable
\[
\mathcal{Q} = [N_{K/\Q}(\gamma^{1/h_2})=1].
\]
We first study the case $\Delta_K \not \in\{-4,-3\}$, which has two subcases: when $\mathcal{Q}=0$, which forces $\Delta_K>0$, and when $\mathcal{Q}=1$. Then, we consider the case $\Delta_K=-4$, and the case $\Delta_K=-3$. We start with the following lemma that generalizes a formula given in the proof of \cite[Lemma 5.4]{Sa2022}:

\begin{lemma}\label{lemma: from sum to product}
    Let $d,h\geq 1$, $e\mid d^\infty$ and $\nu\mid d^\infty$ be positive integers and define
    \begin{equation}\label{equation: definition of S_d,e,h(nu)}
        S_{d,e,h}(\nu) := \sum_{\subalign{v &\mid d^\infty \\ e &\mid v}} \sum_{u\mid d} \frac{\mu(u)(uv,h)[\nu \mid uv]}{\varphi(dv)uv}.
    \end{equation}
    If $d:=D(d,\nu^\infty)$ and $(h,\nu^\infty) \mid \nu$, then
    \[
    S_{d,e,h}(\nu) = \frac{(h,d^\infty)\nu}{d\varphi(\nu)[e,\nu(h,D^\infty)]^2}\cdot \prod_{p\mid \nu} \left( 1- \frac{(pe,\nu)^2}{p(e,\nu)^2}\right) \cdot \prod_{p\mid d} \left( \frac{p^2}{p^2-1} \right).
    \]
\end{lemma}

\begin{proof}
    For $\nu=1$, see the proof of \cite[Lemma 5.4]{Sa2022}. Let us write
    \[
    u=u_1u_2,\quad v=v_1v_2 \quad \text{and} \quad e=e_1e_2,
    \]
    where $u_2 = (u,\nu^\infty)$, $v_2 = (v,\nu^\infty)$, and $e_2 = (e,\nu^\infty)$. We have $[\nu \mid uv] = [\nu \mid u_2v_2]$. Using $\varphi(dv)=\varphi(d)v$ and the multiplicativity of the M\"obius function $\mu$, the Euler totient $\varphi$ and the $\gcd$, we obtain
    \[
    S_{d,e,h}(\nu) =\sum_{\subalign{v_1 &\mid D^\infty \\ e_1&\mid v_1}} \sum_{\subalign{v_2 &\mid \nu^\infty \\ e_2 &\mid v_2 }} \sum_{\subalign{u_1 &\mid D \\ u_2 &\mid \nu}} 
    \left( \frac{\mu(u_1)(u_1v_1,h)}{\varphi(D)u_1v_1^2}\cdot \frac{\mu(u_2)(u_2v_2,h)[\nu\mid u_2v_2]}{\varphi((d,\nu^\infty))u_2v_2^2}\right).
    \]
    Hence $S_{d,e,h}(\nu) = S_{D,e_1,h}(1) \cdot S_{(d,\nu^\infty),e_2,h}(\nu)$. It remains to compute $S_{(d,\nu^\infty),e_2,h}(\nu)$. Since $(h,\nu^\infty)\mid \nu$, we have
    \[
    S_{(d,\nu^\infty),e_2,h}(\nu) = \frac{(h,\nu^\infty)}{\varphi((d,\nu^\infty))} \sum_{\subalign{v_2 &\mid \nu^\infty \\ e_2 &\mid v_2}} \sum_{u_2\mid \nu} \frac{\mu(u_2) [\nu \mid u_2v_2]}{u_2v_2^2},
    \]
    as $\nu \mid u_2v_2$ implies that $(u_2v_2,h) = (h,\nu^\infty)$. We may now interchange the sum and the series to obtain
    \[
    S_{(d,\nu^\infty),e_2,h}(\nu) = \frac{(h,\nu^\infty)}{\varphi((d,\nu^\infty))} \sum_{u_2\mid \nu} \frac{\mu(u_2)}{u_2} \sum_{\subalign{v_2 &\mid \nu^\infty \\ \ell &\mid v_2}} \frac{1}{v_2^2},
    \]
    where $\ell =[e_2, \nu/u_2]$. The inner sum is computed using the Euler product formula:
    \[
    \sum_{\subalign{v_2 &\mid \nu^\infty \\ \ell &\mid v_2}} \frac{1}{v_2^2} = \frac{1}{\ell^2} \sum_{v_2 \mid \nu^\infty } \frac{1}{v_2^2} = \frac{1}{\ell^2} \prod_{p\mid \nu} \sum_{r\geq 0} p^{-2r} = \frac{1}{\ell^2} \prod_{p\mid \nu} \left(\frac{p^2}{p^2-1}\right).
    \]
    We obtain
    \[
    S_{(d,\nu^\infty),e_2,h}(\nu) = \frac{(h,\nu^\infty)}{\varphi((d,\nu^\infty)) } \prod_{p\mid \nu} \left(\frac{p^2}{p^2-1}\right) \cdot \sum_{u_2\mid \nu} \frac{\mu(u_2)}{u_2[e_2,\nu/u_2]^2}.
    \]
    To compute the remaining sum, call it $S$, we expand $[e_2,\nu/u_2]$ and use properties of the $\gcd$. We find that
    \[
    S = \sum_{u_2\mid \nu} \frac{\mu(u_2)(u_2e_2,\nu)^2}{u_2 (e_2 \nu)^2} = \frac{(e_2,\nu)^2}{(e_2\nu)^2} \sum_{u_2\mid \nu} \frac{\mu(u_2)(u_2e_2,\nu)^2}{u_2(e_2,\nu)^2}.
    \]
    The general term of the sum is a multiplicative function in $u_2$, so we apply the Euler product formula again and find
    \[
    S = \frac{(e_2,\nu)^2}{(e_2\nu)^2} \prod_{p\mid \nu} \left(1- \frac{(pe_2,\nu)^2}{p(e_2,\nu)^2} \right) = \frac{(e,\nu)^2}{(e_2\nu)^2} \prod_{p\mid \nu} \left(1- \frac{(pe,\nu)^2}{p(e,\nu)^2} \right),
    \]
    where we used that $e_2 = (e,\nu^\infty)$ on the second equality. Next, applying \cite[Lemma 5.4]{Sa2022} to $S_{D,e_1,h}(1)$, we find that
    \[
    S_{d,e,h}(\nu) = A \cdot \prod_{p\mid d} \left(\frac{p^2}{p^2-1} \right) \cdot  \prod_{p\mid \nu} \left(1- \frac{(pe,\nu)^2}{p(e,\nu)^2} \right),
    \]
    where $A$ is defined by 
    \[
    A := \frac{(h,D^\infty)}{D[(h,D^\infty),e_1]^2} \cdot \frac{(h,\nu^\infty)(e,\nu)^2}{\varphi((d,\nu^\infty))(e_2\nu)^2} =\frac{(h,d^\infty)\nu}{d\varphi(\nu)} \cdot \frac{(e,\nu)^2}{(e_2\nu)^2 [(h,D^\infty),e_1]^2}.
    \]
    We used the following identities in the second equality:
    \[
    \frac{(d,\nu^\infty)}{\varphi((d,\nu^\infty))} = \frac{\nu}{\varphi(\nu)} \quad \text{and} \quad (h,D^\infty)(h,\nu^\infty) = (h,d^\infty).
    \]
    Finally, we can expand the $\mathrm{lcm}$ to find
    \[
    \frac{(e,\nu)^2}{(e_2\nu)^2 [(h,D^\infty),e_1]^2} = \frac{(e,\nu)^2 (h,e_1,D^\infty)^2}{(e_2\nu)^2 (h,D^\infty)^2e_1^2} =  \frac{(e,\nu)^2 (h,e,D^\infty)^2}{\nu^2 (h,D^\infty)^2e^2}. 
    \]
    Since $\nu$ and $(h,D^\infty)$ are coprime, we have $(e,\nu)(h,e,D^\infty) = (e,\nu(h,D^\infty))$. We see that the formula becomes
    \[
    \frac{ (e,\nu(h,D^\infty))^2}{\nu^2 (h,D^\infty)^2e^2} = \frac{1}{[e,\nu(h,D^\infty)]^2},
    \]
    and the result follows.
\end{proof}

\begin{remark}\label{remark: S_d,e,h = 0}
    Note that $S_{d,eh}(\nu)$, as defined in \eqref{equation: definition of S_d,e,h(nu)}, is zero if $e\nmid d^\infty$ or $\nu \nmid d^\infty$.
\end{remark}

In the rest of the paper, our goal is to write $\delta_\gamma^+(d)$ and $\delta_\gamma^-(d)$ as linear combinations of the sums $S_{d,e,h}(\nu)$. Then, the density values will be in closed-forms by Lemma \ref{lemma: from sum to product} and Remark \ref{remark: S_d,e,h = 0}. We write $S_{d,e,h}$ as a shorthand for $S_{d,e,h}(1)$.

\subsection{The case \texorpdfstring{$\Delta_K\not \in\{-4,-3\}$ and $\mathcal{Q}=0$}{TEXT}}\label{subsection: the case Q=0}

First, note that $h_2\geq 2$ because $N_{L/\Q}(\gamma)=-1$. Moreover, if $\gamma^{1/h_2} = u+v\sqrt{\Delta_K}$ for some $u,v\in \Q^\times$, then $\mathcal{Q}=0$ is equivalent to $u^2-v^2\Delta_K = -1$. Therefore, the assumption $\mathcal{Q}=0$ ensures that $2\mid h_2$ and $\Delta_K>0$. These facts are used to find a closed-form of $\delta_\gamma^-(d)$.

\begin{theorem}\label{theorem: delta+(d) with Q=0}
    Assume $\mathcal{Q}=0$ and $2\mid d$. Let $e=\Delta_K/(d,\Delta_K)$. Then,
    \[
    \delta_\gamma^+(d) = \frac{1}{2d} \left( \frac{1}{(h,d^\infty)} + [e\mid d^\infty]\cdot \frac{(h,d^\infty)}{[(h,d^\infty),e]^2} \right) \prod_{p\mid d}\left(\frac{p^2}{p^2-1}\right).
    \]
\end{theorem}

\begin{proof}
    Since $\mathcal{Q}=0$, and thus $N_{K/\Q}(\gamma^{1/h_2})=-1$, we see that $\gamma^{1/h_2}\not \in (K(\zeta_n)^\times)^2$ by Lemma \ref{lemma: square root in cyclotomic extensions}. By Corollary \ref{corollary: degree of Kummer extensions}, we have
    \begin{equation}\label{Equation, degree Kdv,uv/Q when Q=0}
        [K_{dv,uv}:\Q] = \frac{\varphi(dv)uv}{(uv,h)} \cdot 2^{[\Delta_K\nmid dv]}.
    \end{equation}
    Hence, because $\Delta_K \mid dv$ if and only if $e\mid v$, we obtain
    \[
    \delta_\gamma^+(d) = \sum_{v\mid d^\infty} \sum_{u\mid d} \frac{\mu(u)(uv,h)}{\varphi(dv)uv} \cdot \frac{1}{2^{[e \nmid v]}} = \frac{S_{d,1,h}  +S_{d,e,h}}{2}.
    \]
    Note that we used the identity $2^{1-[e\nmid v]} = 1+[e\mid v]$ in the last equality. The result follows by Lemma \ref{lemma: from sum to product} and Remark \ref{remark: S_d,e,h = 0}.
\end{proof}

\begin{theorem}\label{theorem: delta-(d) with Q=0}
    Assume $\mathcal{Q}=0$ and $2\mid d$. Let $e=\Delta_K/(d,\Delta_K)$. Then,
    \[
    \delta_\gamma^-(d) = \frac{3}{2d}\left( \frac{1}{(h,d^\infty)} -[ e\mid d^\infty \text{ and } h_2\nmid e ]\cdot \frac{(h,d^\infty)}{[(h,d^\infty),e]^2} \right) \prod_{p\mid d} \left( \frac{p^2}{p^2-1}\right).
    \]
\end{theorem}

\begin{proof}
    Recall that $\mathcal{Q}=0$ implies $2\mid h$ and $\Delta_K>0$. By Theorem \ref{theorem: existence of sigma_u,v}, we know that $\sigma_{u,v}$ exists if and only if $h_2\nmid uv$ and $\Delta_K \nmid dv$. Using Corollary \ref{corollary: degree of Kummer extensions}, we obtain
    \[
    \delta_\gamma^-(d) = \sum_{v\mid d^\infty} \sum_{u \mid d} \frac{\mu(u)(uv,h)}{\varphi(dv)uv} \cdot \frac{[ e\nmid v] [h_2\nmid uv]}{2^{[e\nmid v]}}.
    \]
    We linearize $\delta_\gamma^-(d)$ using
    \[
    \frac{[e\nmid v] [h_2\nmid uv]}{2^{[e\nmid v]}} = \frac{1-[h_2\mid uv]-[e\mid v]+[e\mid v][h_2\mid uv]}{2},
    \]
    so that, with the notation of Lemma \ref{lemma: from sum to product}, we obtain
    \[
    \delta_\gamma^-(d) = \frac{1}{2}\bigg( S_{d,1,h} -S_{d,1,h}(h_2) -S_{d,e,h} +S_{d,e,h}(h_2)  \bigg).
    \]
    By Lemma \ref{lemma: from sum to product} with $\nu = h_2$, we see that $S_{d,1,h}(\nu) = -2S_{d,1,h}$. Similarly, when $e\mid d^\infty$, we have $S_{d,e,h}(\nu ) = \big(1-3\cdot[h_2\nmid e] \big) \cdot S_{d,e,h}$. Hence
    \[
    \delta_\gamma^-(d) = \frac{3}{2}\bigg(S_{d,1,h} - [h_2\nmid e]\cdot S_{d,e,h} \bigg),
    \]
    and the result follows by Lemma \ref{lemma: from sum to product} and Remark \ref{remark: S_d,e,h = 0}.
\end{proof}

\subsection{The case \texorpdfstring{$\Delta_K\not \in\{-4,-3\}$ and $\mathcal{Q}=1$}{TEXT}}\label{subsection: the case Q=1}\

To compute $t$ (defined in Theorem \ref{theorem: minimal polynomial when h=h(1)}), we have to check whether $\gamma^{1/h_2}$ is a square or not in $K(\zeta_n)$. Let $\gamma^{1/h_2} = u+v\sqrt{\Delta_K}$, where $u,v\in\Q^\times$. Let $c=(u-1)/2$, and $\Delta_1$ and $\Delta_2$ be the absolute discriminants of $K_1=\Q(\sqrt{c})$ and $K_2=\Q(\sqrt{c/\Delta_K})$ respectively. By Lemma \ref{lemma: square root in cyclotomic extensions}, we know that $\gamma^{1/h_2}\in (K(\zeta_n)^\times)^2$ if and only if $\Delta_1\mid n$ or $\Delta_2\mid n$. Moreover, we have
\begin{equation}\label{equation: sqrt(r)}
    \gamma^{1/2h_2} = \sqrt{c} + \frac{v}{2} \sqrt{\frac{\Delta_K}{c}}.
\end{equation}
Note that $c=0$ only if $\gamma=1$, which contradicts that $\gamma$ is not a root of unity. In addition, we have $[\Delta_K,\Delta_1] = [\Delta_K,\Delta_2] = [\Delta_1,\Delta_2]$ because $K$, $K_1$ and $K_2$ are pairwise linearly disjoint over $\Q$ with $KK_1 = KK_2 = K_1K_2$.

\begin{theorem}\label{theorem: delta+(d) with Q=1}
    Assume $\mathcal{Q}=1$ and $2\mid d$. Define
    \[
    e = \frac{|\Delta_K|}{(d,|\Delta_K|)} \quad \text{and} \quad e_i = \frac{|\Delta_i|}{(d,|\Delta_i|)},
    \]
    for all  $1\leq i \leq 2$. Then, 
    \[
    \delta_\gamma^+(d) =\frac{1}{2} \bigg( S_{d,1,h}+S_{d,e,h} + S_{d,e_1,h}(2h_2)+ S_{d,e_2,h}(2h_2) \bigg).
    \]
    In particular, if $\Delta_K <0$, then $\delta_\gamma(d) = 2\delta_\gamma^+(d)$.
\end{theorem}

\begin{proof}
    By Corollary \ref{corollary: degree of Kummer extensions}, we have
    \[
    \delta_\gamma^+(d) =\sum_{v\mid d^\infty}\sum_{u\mid d}\frac{\mu(u)(uv,h)}{\varphi(dv)uv}\cdot \frac{t}{2^{[\Delta_K\nmid dv]}}.
    \]
    Then, we linearize $t\cdot 2^{-[\Delta_K\mid dv]}$ using the identities $2^{1-[\Delta_K\nmid dv]} =1+[\Delta_K\mid dv]$ and $t=1+[2h_2\mid uv]\cdot [\Delta_1\mid dv \text{ or } \Delta_2\mid dv]$, which is obtained from Lemma \ref{lemma: square root in cyclotomic extensions}. Next, we use the inclusion-exclusion principle on $t$, the equivalence $\Delta_K\mid dv$ if and only if $e\mid v$, which is also valid for $(\Delta_1,e_1)$ and $(\Delta_2,e_2)$, and the equalities $[e,e_1] = [e,e_2] = [e_2,e_3]$ discussed at the beginning of Subsection \ref{subsection: the case Q=1}. We obtain
    \[
    \frac{t}{2^{[\Delta_K\nmid dv]}} = \frac{1+[e\mid v]}{2} + \frac{[2h_2\mid uv] \cdot [e_1\mid v]}{2} + \frac{[2h_2\mid uv] \cdot [e_2\mid v]}{2}.
    \]
    We may now linearize $\delta_\gamma^+(d)$ using the above to find the result. Finally, if $\Delta_K<0$, then $\delta_\gamma^+(d) = \delta_\gamma^-(d)$ by Lemma \ref{lemma: existence of sigma_u,v when D<0}.
\end{proof}

We now turn our attention to $\delta_\gamma^-(d)$ in the case $\Delta_K>0$. Let us first revise the existence conditions for the automorphism $\sigma\in \Gal(K_{n,d}/\Q)$ defined in Theorem \ref{theorem: existence of sigma_u,v}. Since $\mathcal{Q}=1$, we see that $\sigma$ exists if and only if
\begin{enumerate}[label=(\arabic*)]\setlength{\itemsep}{0.5em}
    \item $\Delta_K \nmid n$;

    \item $\mathcal{P}_1(n,d)=1$ or $\mathcal{P}_2(n,d)=1$,
\end{enumerate}\smallskip
with Boolean functions
\begin{equation}\label{equation: definition of P_1(n,d)}
    \mathcal{P}_1(n,d) = [2h_2\nmid d \text{ or } \forall i\in \{1,2\}, \ \Delta_i \nmid n],
\end{equation}
and
\begin{equation}\label{equation: definition of P_2(n,d)}
    \mathcal{P}_2(n,d) = [2h_2\mid d]\cdot[\Delta_1\mid n \text{ or } \Delta_2\mid n]\cdot [ \sigma_0(\gamma^{1/2h_2}) = \gamma^{-1/2h_2}],
\end{equation}
where $\sigma_0\in \Gal(K(\zeta_n)/K)$ is such that $\sigma_0(\sqrt{\Delta}) = -\sqrt{\Delta}$ and $\sigma_0(\zeta_n) = \zeta_n^{-1}$. The next lemma provides a way to compute $ \sigma_0(\gamma^{1/2h_2})$ without having to consider $K(\zeta_n)$.

\begin{lemma}\label{lemma: P_2(n,d) other form}
    Assume $\mathcal{Q}=1$, $\Delta_K>0$ and $\Delta_K\nmid n$. Then, we have
    \[
    \mathcal{P}_2(n,d) = [2h_2\mid d]\cdot
    \begin{cases}
        [\Delta_1 \mid n], & \text{if $c<0$}; \\
        [\Delta_2 \mid n], & \text{if $c>0$},
    \end{cases}
    \]
    where $c$ is defined in \eqref{equation: sqrt(r)}.
\end{lemma}

\begin{proof}
    By the proof of \cite[Lemma 4.2]{Sa2022}, we know $\sigma_0\in \Gal(K(\zeta_n)/K)$ because $\Delta_K\nmid n$. Moreover, we only  work on the Boolean variable $[\sigma_0(\gamma^{1/2h_2}) = \gamma^{-1/2h_2}]$ that appears in $\mathcal{P}_2(n,d)$. Thus, we may assume that $2h_2\mid d$ and one of $\Delta_1$ or $\Delta_2$ divides $n$. The latter ensures the existence of a square root of $\gamma^{1/h_2}$ in $K(\zeta_n)$. Note that $[\Delta_1,\Delta_2] \nmid n$, otherwise $\Delta_K$ would divide $n$. From \eqref{equation: sqrt(r)}, we find
    \[
    \gamma^{-1/2h_1} = -\sqrt{c} + \frac{v}{2}\sqrt{\frac{\Delta_K}{c}}.
    \]
    By comparing them, we see that it suffices to find conditions to have $\sigma_0(\sqrt{c}) = -\sqrt{c}$ and $\sigma_0(\sqrt{c/\Delta_K}) = \sqrt{c/\Delta_K}$.

    Assume that $\Delta_1 \mid n$. By definition, we have $\sqrt{c} \in \Q(\zeta_n)$. Since $\sigma_0|_{\Q(\zeta_n)}$ is the complex conjugation, for $\sqrt{c}$ to be sent to $-\sqrt{c}$, we need $c<0$. We obtain
    \[
    \sigma_0\left(\sqrt{\frac{c}{\Delta_K}} \right)  = \sqrt{\frac{c}{\Delta_K}},
    \]
    since $\sigma_0(\sqrt{\Delta_K}) = -\sqrt{\Delta_K}$. Hence $\sigma_0(\gamma^{1/2h_2}) = \gamma^{-1/2h_2}$. The case  $\Delta_2\mid n$ proceeds similarly, with $c/\Delta_K$ instead of $c$.
\end{proof}

\begin{theorem}\label{theorem: delta-(d) with Q=1}
    Assume $\mathcal{Q}=1$, $\Delta_K>0$, and $2\mid d$. Define
    \[
    e = \frac{\Delta_K}{(d,\Delta_K)} \quad \text{and} \quad e_i = \frac{|\Delta_i|}{(d,|\Delta_i|)},
    \]
    for all $1\leq i \leq 2$. Then,
    \[
    \delta_\gamma^-(d) = \frac{S_{d,1,h} -S_{d,e,h}}{2} + (-1)^{[c>0]} \cdot\frac{S_{d,e_1,h}(2h_2) -S_{d,e_2,h}(2h_2)}{2},
    \]
    where $c$ is defined in \eqref{equation: sqrt(r)}.
\end{theorem}

\begin{proof}
    Let $\nu = 2h_2$. By \eqref{equation: definition of P_1(n,d)} and \eqref{equation: definition of P_2(n,d)}, we have $\big(\mathcal{P}_1(dv,uv) + \mathcal{P}_2(dv,uv) \big)\cdot [\Delta_K\nmid dv] $ equals $1$ if and only if $\sigma_{u,v}$ exists. We may write $\delta_\gamma^-(d) = S_1(d) + S_2(d)$, where
    \[
    S_i(d) =  \sum_{v\mid d^\infty} \sum_{u\mid d} \frac{\mu(u)\mathcal{P}_i(dv,uv)}{[K_{dv,uv}:\Q]} \cdot [ \Delta_K \nmid dv],
    \]
    for all $i\in\{1,2\}$. We start with $S_1(d)$. Assuming $\mathcal{P}_1(dv,uv)=1$ and $\Delta_K\nmid dv$, and using Corollary \ref{corollary: degree of Kummer extensions}, we obtain a general term for $S_1(d)$ of 
    \[
    \frac{\mu(u)(uv,h)}{2\varphi(dv)uv}  \cdot \mathcal{P}_1(dv,uv)\cdot [e\nmid v].
    \]
    We used $\Delta_K \mid dv$ if and only if $e\mid v$, which is also valid for $(e_1,\Delta_1)$ and $(e_2,\Delta_2)$. Next, we may write $\mathcal{P}_1(dv,uv) = 1-[2h_1 \mid uv \text{ and } \exists i\in \{1,2\}, \ \Delta_i \mid dv]$ and, because of the equalities $[e,e_1]=[e,e_2]=[e_1,e_2]$, we find that
    \[
    \mathcal{P}_1(dv,uv) \cdot [e\nmid v] = \big(1 - [2h_2\mid uv]\cdot [e_1\mid v] - [2h_2\mid uv]\cdot [e_2\mid v]\big) \cdot [e\nmid v].
    \]
    We can now use $[e\nmid v] = 1-[e\mid v]$ to obtain
    \[
    S_1(d) = \frac{1}{2}\bigg( S_{d,1,h} - S_{d,e,h} -S_{d,e_1,h}(2h_2) -S_{d,e_1,h}(2h_2) \bigg) + S_{d,[e_1,e_2],h}(2h_2).
    \]
    We now turn our attention to $S_2(d)$. By Lemma \ref{lemma: P_2(n,d) other form}, we have
    \[
    \mathcal{P}_2(n,d)\cdot [\Delta_K\nmid dv] = [2h_2\mid uv] \cdot [\Delta_K\nmid dv]\cdot
    \begin{cases}
        [\Delta_1 \mid dv ], & \text{if $c<0$}; \\
        [\Delta_2 \mid dv ], & \text{if $c>0$},
    \end{cases}
    \]
    If $c<0$, then
    \[
    S_2(d) = \sum_{\subalign{v &\mid d^\infty \\ e_1 &\mid v}} \sum_{u\mid d} \frac{\mu(u)(uv,h) t}{2\varphi(dv)uv} \cdot [2h_2\mid uv]\cdot[e\nmid v] = S_{d,e_1,h}(\nu) -S_{d,[e,e_1],h}(\nu),
    \]
    using Corollary \ref{corollary: degree of Kummer extensions} for the field degrees and Lemma \ref{lemma: square root in cyclotomic extensions} to show that $t=2$, where $t$ is defined in the corollary. Going back to $\delta_\gamma^-(d) = S_1(d)+S_2(d)$, we obtain
    \[
    \delta_\gamma^-(d) = \frac{1}{2} \bigg(S_{d,1,h} -S_{d,e,h} -S_{d,e_1,h}(\nu) +S_{d,e_2,h}(\nu)\bigg),
    \]
    where we used $[e,e_1]=[e_1,e_2]$. If $c>0$, then $S_2(d) = S_{d,e_2,h}(\nu) - S_{d,[e,e_2],h}(\nu)$, and the result follows similarly.
\end{proof}

\subsection{The case \texorpdfstring{$\Delta_K=-4$}{TEXT}}

We now deal with the case of $K=\Q(i)$. Similar to Subsection \ref{subsection: the case Q=1}, we have to check if $\gamma^{1/h_2}$ is a square in cyclotomic extensions $K(\zeta_n)$ or not. Because $N_{K/\Q}(\gamma^{1/h_2}) = 1$, it is a square if and only if $\Delta_1 \mid n$ or $\Delta_2 \mid n$, where the $\Delta_i$'s were defined in the previous subsection. However, when it is a square, we also have to check if $\gamma^{1/h_2}$ is a $4$-th power in $K(\zeta_n)$. By Lemma \ref{lemma: fourth root in cyclotomic fields}, if we consider $\mathfrak{f}(L)$, the conductor of $L:=\Q(i,\gamma^{1/4h_2})$, then $\gamma^{1/4h_2} \in K(\zeta_n)$ if and only if $\mathfrak{f}(L)\mid [4,n]$.

\begin{theorem}\label{theorem: delta(d) with Delta=-4}
    Assume that $\Delta_K=-4$ and $2\mid d$. Define $e=4/(d,4)$,
    \[
    f = \frac{\mathfrak{f}(L)}{(d,\mathfrak{f}(L))} \quad \text{and} \quad e_i = \frac{|\Delta_i|}{(d,|\Delta_i|)},
    \]
    for all $1\leq i\leq 2$. Then,
    \[
    \delta_\gamma(d) = S_{d,1,h}+S_{d,e,h}+S_{d,e_1,h}(2h_2)+S_{d,e_2,h}(2h_2)+4S_{d,f,h}(4h_2).
    \]
\end{theorem}

\begin{proof}
    The proof is similar to those of Theorems \ref{theorem: delta+(d) with Q=0}, \ref{theorem: delta-(d) with Q=0}, \ref{theorem: delta+(d) with Q=1}, and \ref{theorem: delta-(d) with Q=1}, thus we only give the main steps. First, we linearize $t$ defined in Theorem \ref{theorem: minimal polynomial when h=h(1)} as
    \[
    t = 1 + [2h_2\mid uv]\cdot [\exists i, \Delta_i\mid dv] + 2 [4h_2 \mid uv]\cdot [\mathfrak{f}(L)\mid dv]
    \]
    for the field $K_{dv,uv}$. Note that we used $4h_2\mid uv$ and $u$ being square-free, so that $2\mid v$. Hence $4\mid dv$ and $[4,dv]=dv$. Secondly, by Corollary \ref{corollary: degree of Kummer extensions}, we have
    \[
    \frac{1}{[K_{dv,uv} :\Q]} = \frac{(uv,h)t}{\varphi(dv)uv} \cdot \frac{1}{2^{[4\nmid dv]}} = \frac{(uv,h)t}{\varphi(dv)uv} \cdot  \frac{1+[4\mid dv]}{2}.
    \]
    Finally, replacing this expression in $\delta_\gamma(d)$, we can linearize $t$ and write $\delta_\gamma(d)$ in terms of sums $S_{d,e,h}(\nu)$. Also, note that we use $\Delta_i \mid dv$ if and only if $e_i\mid v$, and the same holds for the pairs $(4,e)$ and $(\mathfrak{f}(L),f)$, and $\delta_\gamma^+(d) = \delta_\gamma^-(d)$ by Lemma \ref{lemma: existence of sigma_u,v when D<0}.
\end{proof}

\subsection{The case \texorpdfstring{$\Delta_K=-3$}{TEXT}}

This is the last case for $h=h(1)$. As before, we check whether $\gamma^{1/h_2}$ has a square root in $K(\zeta_n)$ or not, using the discriminants $\Delta_1$ and $\Delta_2$, which we defined as in Subsection \ref{subsection: the case Q=1}. In addition, we have to check whether $\gamma^{1/h_3}$ is a cube in $K(\zeta_n)$ or not. Let us write $\gamma^{1/h_3} = u+v\sqrt{-3}$ for some $u,v\in \Q^\times$, and let $L$ be the splitting field of $X^3-3X-2u$. If $\mathfrak{f}(L)$ is the conductor of $L$, then $\gamma^{1/h_3} \in (K(\zeta_n)^\times)^3$ if and only if $\mathfrak{f}(L) \mid n$, by Lemma \ref{lemma: cube root in Q(zeta_3,zeta_n)}.

\begin{theorem}\label{theorem: delta(d) with Delta=-3}
    Assume that $\Delta_K=-3$ and $(d,6)>1$. Define
    \[
    f = \frac{\mathfrak{f}(L)}{(d,\mathfrak{f}(L))} \quad \text{and} \quad \Tilde{e} = \min\left(\frac{|\Delta_i|}{(d,|\Delta_i|)}  : 1\leq i\leq 2\right).
    \]
    Then, we have
    \[
    \delta_\gamma(d) = 2^{[3\mid d]} \bigg(S_{d,1,h} + S_{d,\Tilde{e},h}(2h_2) + 2S_{d,f,h}(3h_3) + 2S_{d,[\Tilde{e},f],h}(6h_6) \bigg).
    \]
\end{theorem}

\begin{proof}
    As in the proof of Theorem \ref{theorem: delta(d) with Delta=-4}, we may skip a few details. Let $t$ be defined as in Theorem \ref{theorem: minimal polynomial when h=h(1)} applied to the field $K_{dv,uv}$. Then, we have
    \[
    t = \big(1+[2h_2\mid uv]\cdot [\exists i, \Delta_i\mid dv] \big)\cdot \big(1+2\cdot [3h_3\mid uv]\cdot[\mathfrak{f}(L) \mid dv]\big),
    \]
    Moreover, the fields $K_1=\Q(\sqrt{c})$ and $K_2=\Q(\sqrt{-3c})$ have discriminants $\Delta_1$ and $\Delta_2$, which are equal up to a factor of $3$. If $3\nmid d$, then it means only one of $\Delta_1$ or $\Delta_2$ may divide $dv$. If $3\mid d$, then $\Delta_1 \mid dv$ if and only if $\Delta_2 \mid dv$. In both cases, we obtain
    \[
    [\exists i, \Delta_i \mid dv] = [\Tilde{e} \mid v],
    \]
    which simplifies the expression of $t$. Next, by Corollary \ref{corollary: degree of Kummer extensions}, we have
    \[
    [K_{dv,uv}:\Q] = \frac{\varphi(dv)uv}{(uv,h)t}\cdot 2^{[3\mid dv]} =  \frac{\varphi(dv)uv}{(uv,h)t}\cdot 2^{[3\mid d]},
    \]
    because $v\mid d^\infty$. Finally, replacing this degree in the expression of $\delta_\gamma^+(d)$, linearizing $t$, and using Lemma \ref{lemma: existence of sigma_u,v when D<0} for the equality $\delta_\gamma^+(d)=\delta_\gamma^-(d)$, we can write $\delta_\gamma(d)$ as the linear combination of the sums $S_{d,e,h}(\nu)$ given in the statement.
\end{proof}

\section{Remaining Cases}\label{section: remaining cases}

Our goal is to provide a way to write $\delta_\gamma(d)$ in closed-form when $h\ne h(1)$. We first take a quick look at the case $h=h(-1)$. Then, we study the cases $h=h(\zeta)$ for $\zeta\in \mu(K)$ different from $\pm 1$. This may only happen when $\Delta_K=-4$ or $\Delta_K=-3$. The next formula allows us to write $\delta_\gamma(d)$ solely in terms of $\delta_{-\gamma}(\cdot)$, which can be written in closed-form by Section \ref{section: h=h(1)}. This is done via a formula used by Wiertelak \cite{WiertelakIV} and Moree \cite{Mor2005} that links $\ord_\pi(\gamma)$ and $\ord_\pi(-\gamma)$. We write $\zeta\gamma = \gamma_0^h$ for some $\gamma_0\in K$.

\begin{theorem}\label{theorem: switch with -1}
    For every $d\geq 2$, we have
    \[
    \delta_\gamma(d) = 
    \begin{cases}
        \delta_{-\gamma}(2d)+\delta_{-\gamma}(d/2)-\delta_{-\gamma}(d), & \text{if $v_2(d)=1$}; \\
        \delta_{-\gamma}(d), & \text{otherwise}.
    \end{cases}
    \]
    In particular, if $h=h(-1)$, then the values for $\delta_{-\gamma}$ can be computed using Theorems \ref{theorem: delta+(d) with Q=0}, \ref{theorem: delta-(d) with Q=0}, \ref{theorem: delta+(d) with Q=1}, \ref{theorem: delta-(d) with Q=1}, \ref{theorem: delta(d) with Delta=-4}, or \ref{theorem: delta(d) with Delta=-3}.
\end{theorem}

\begin{proof}
    Let $p \in \mathcal{R}_\gamma(d)$ and $\pi$ be a prime ideal of $\mathcal{O}_K$ lying above $p$. We have
    \[
    \ord_\pi(\gamma) =
    \begin{cases}
        \ord_\pi(-\gamma)/2, & \text{if $2\nmid \ord_\pi(\gamma)$}; \\
        2\ord_\pi(-\gamma), & \text{if $v_2( \ord_\pi(\gamma))=1$}; \\
        \ord_\pi(-\gamma), & \text{if $4\mid \ord_\pi(\gamma)$}.
    \end{cases}
    \] 
    We see that $\ord_\pi(\gamma) = \mathrm{ord}_\pi(-\gamma)$ if $4\mid d$ and $d\mid \ord_\pi(\gamma)$. Moreover, if $2\nmid d$, then $d\mid \ord_\pi(\gamma)$ if and only if $d\mid \mathrm{ord}_\pi(-\gamma)$.  Hence $\mathcal{R}_\gamma(d) = \mathcal{R}_{-\gamma}(d)$ and the result follows if $v_2(d)\ne 1$. Assume $v_2(d)=1$. Then, $\mathcal{R}_\gamma(d)$ is the union of the sets
    \[
    A_1= \{p \text{ prime} : p\nmid a_2\Delta \text{ and } 2d\mid \ord_\pi(\gamma)\},
    \]
    and $A_2= \mathcal{R}_\gamma(d)\setminus A_1$. Any $p\in A_1$ satisfies $4\mid \ord_\pi(\gamma)$, so that $\ord_\pi(\gamma)=\mathrm{ord}_\pi(-\gamma)$. Thus, we have $A_1=\mathcal{R}_{-\gamma}(2d)$. For $A_2$, we see that any $p$ in this set satisfies
    \[
    d\mid \ord_\pi(\gamma) \quad \text{and} \quad 2d\nmid \ord_\pi(\gamma).
    \]
    It follows that $\ord_\pi(\gamma) = 2\mathrm{ord}_\pi(-\gamma)$, and that $p\in A_2$ if and only if $d \mid 2\mathrm{ord}_\pi(-\gamma) $ and $ d\nmid \mathrm{ord}_\pi(-\gamma)$. We find that $A_2 = \mathcal{R}_{-\gamma}(d/2)\setminus \mathcal{R}_{-\gamma}(d)$. The result follows by taking the natural density of these sets, which exists by Sanna's theorem \cite{Sa2022}.
\end{proof}

Note that these switching formulas between $\gamma$ and $-\gamma$ given in Theorem \ref{theorem: switch with -1} remain valid for $\delta_\gamma^+$ and $\delta_\gamma^-$, and that the proof is the same. We also use the following:

\begin{lemma}\label{lemma: remaining cases degree}
    Assume that $\Delta_K\in\{-4,-3\}$. Let $p\mid \Delta_K$ be prime and $d\geq 2$, $p\nmid d$. Then, we have
    \[
    [K_{p^{k+j}dv,uv}:\Q] = [K_{dv,uv}:\Q] \cdot 
    \begin{cases}
        3^{k+j-1}, & \text{if $p=3$}; \\
        2^{k+j-2}, & \text{if $p=2$ and $k+j\geq 2$}; \\
        1, & \text{if $p=2$ and $(k,j) = (1,0)$},
    \end{cases}
    \]
    for all $u\mid d$, $v\mid d^\infty$, $k\geq 1$, and $j\geq 0$.
\end{lemma}

\begin{proof}
    It suffices to see that $K_{p^{k+j}dv,uv}$ is the compositum of $K_{dv,uv}$ and $K(\zeta_{p^{k+j}})$, which are linearly disjoint over $K$. Then, we use the formula
    \[
    [K_{p^{k+j}dv,uv}:\Q] = [K_{dv,uv}: \Q] \cdot [K(\zeta_{p^{k+j}}) : K],
    \]
    for composita, and compute the degree of the cyclotomic extension $K(\zeta_{p^{k+j}})/K$.
\end{proof}

\subsection{The case \texorpdfstring{$h\ne h(1)$ and $\Delta_K=-4$}{TEXT}}

We may assume that $h=h(i) \ne h(\pm1)$. Indeed, Theorem \ref{theorem: switch with -1} already addresses the case $h=h(-1)$. Moreover, it allows us to switch between $\delta_\gamma$ and $\delta_{-\gamma}$. Thus, we may assume $h=h(i)$ in the following. Note that $2\mid h$ because $h\ne h(\pm 1)$.

\begin{theorem}\label{theorem: Delta=-4 et h=h(i)}
    Assume $\Delta_K = -4$ and $h =h(i)\ne h(\pm1)$. Let $d,k\geq 1$ be integers, $2\nmid d$, and $\Delta_1$ and $\mathfrak{f}(L)$ be as defined in Lemmas \ref{lemma: square root in cyclotomic extensions} and \ref{lemma: fourth root in cyclotomic fields} for $\gamma_0^{h/h_2}$. Then,
    \[
    \delta_\gamma(2^kd) = \delta_{i\gamma}(d) \cdot 
    \begin{cases}
        \displaystyle 1- \frac{2^k}{3\cdot 2^{m+2}h_2}, & \text{if $k\in\{1,2\}$};\\[0.7em]

        \displaystyle \frac{8}{3\cdot 2^{k+m}h_2}, & \text{if $k\geq 3$},
    \end{cases}
    \]
    where $m=[\Delta_1 \mid 8d] + [\mathfrak{f}(L) \mid 16d]$ and $\delta_{i\gamma}(d)$ is computed in \cite[Theorem 1.1]{Sa2022}.
\end{theorem}

\begin{proof}
    We give the main steps of the proof. First, by Lemma \ref{lemma: existence of sigma_u,v when D<0}, we can write
    \begin{equation}\label{equation: delta(2^kd) as four sums}
        \delta_\gamma(2^kd) = \sum_{j\geq 0} \sum_{v\mid d^\infty} \sum_{0\leq l \leq 1} \sum_{u\mid d} \frac{2\mu(2^lu)}{[K_{2^{k+j}dv,2^{j+l}uv}:\Q]}.
    \end{equation}
    Then, by Kummer theory, we see that $K_{2^{k+j}dv,uv}$ and $K_{2^{k+j}dv,2^{j+l}}$ are linearly disjoint over $K(\zeta_{2^{k+j}dv})$. By the formula for the degree of composita, we have
    \[
    [K_{2^{k+j}dv,2^{j+l}uv}:\Q] = [K_{2^{k+j}dv,uv}:\Q] \cdot [K_{2^{k+j}dv,2^{j+l}}:K(\zeta_{2^{k+j}dv})].
    \]
    The degree of $K_{2^{k+j}dv,uv}/\Q $ is computed in Lemma \ref{lemma: remaining cases degree}. To compute the second degree, we need to find the largest $n\geq 1$ such that $\gamma_0^{h} =x^{2^n}$ in $K(\zeta_{2^{k+j}dv})$. The degree is easily computed when $k+j\leq 3$. If $k+j\geq 4$, then $\gamma_0^{h/h_2}$ is a square, respectively a $4$-th power, in $K(\zeta_{2^{k+j}dv})$ if and only if $\Delta_1\mid 2^{k+j}dv$, respectively $\mathfrak{f}(L) \mid 2^{k+j}dv$, by Lemmas \ref{lemma: square root in cyclotomic extensions} and \ref{lemma: fourth root in cyclotomic fields}. Because $\Delta_1$ and $\mathfrak{f}(L)$ are square-free except possibly for a factor $2^n$, where $n\leq 3$ for $\Delta_1$ and $n\leq 4$ for $\mathfrak{f}(L)$, these conditions become $\Delta_1\mid 8d$ and $\mathfrak{f}(L)\mid 16d$. Hence the definition of $m$. If $k+j\geq 4$, we have
    \[
    [K_{2^{k+j}dv,2^{j+l}}:K(\zeta_{2^{k+j}dv})] = 
    \begin{cases}
        2^{l+2-k}, & \text{if $k\in\{1,2\}$}; \\
        1, & \text{if $k\geq 3$},
    \end{cases}
    \]
    when $j+l \leq v_2(2^mh)$. Otherwise, when $j+l>v_2(2^mh)$, we have
    \[
    [K_{2^{k+j}dv,2^{j+l}}:K(\zeta_{2^{k+j}dv})] = 
    \begin{cases}
        [2^{l+1}, 2^{j+l-v_2(2^mh)}], & \text{if $k=1$}; \\
        2^{j+l-v_2(2^mh)}, & \text{if $k\geq 2$}.
    \end{cases}
    \]
    In the latter case, we first used \cite[Lemma 4]{Perucca2015} with the fields $K(\zeta_{2^{k+j}dv}, (-i)^{1/2^{j+l}})$ and $K(\zeta_{2^{k+j}dv}, \gamma_0^{h_0/2^{j+l}})$, which are linearly disjoint over $K(\zeta_{2^{k+j}dv})$. Then, we used Theorem \ref{theorem: irreducibility of X^n-a} to compute the degree of $K(\zeta_{2^{k+j}dv}, \gamma_0^{h_0/2^{j+l}})$. In particular, we find that the degree $A_{j,k,l} := [K_{2^{k+j}dv,2^{j+l}}:K(\zeta_{2^{k+j}dv})]$ is independent of $u$ and $v$. Using Lemma \ref{lemma: remaining cases degree}, we obtain
    \[
    \delta_\gamma(2^kd) = \delta_\gamma(d) \cdot \sum_{j\geq 0} \bigg( \frac{1}{A_{j,k,0}} - \frac{1}{A_{j,k,1}}\bigg)  \cdot 
    \begin{cases}
        1, & \text{if $k=1$ and $j=0$};\\
        2^{2-j-k}, & \text{if $k+j\geq 2$}.
    \end{cases}
    \]
    We have now obtained $\delta_\gamma(d)$ multiplied by a geometric series, the computation of which yields the claimed formula. 
    
    Finally, note that $K(\zeta_{dv}, \gamma^{1/uv})$ and $K({\zeta_{dv}}, \gamma_0^{h/uv})$ are equal for all $u\mid d$ and $v\mid d^\infty$. Indeed, since $d$ is odd, we have $(-i)^{1/uv}=\pm i$. Therefore, $\gamma^{1/uv} = \pm i\gamma_0^{h/uv}$ and $ \delta_\gamma(d) = \delta_{i\gamma}(d)$ for all odd $d\geq 1$, which is computed in \cite[Theorem 1.1]{Sa2022}.
\end{proof}

\subsection{The case \texorpdfstring{$h\ne h(1)$ and $\Delta_K=-3$}{TEXT}}

We compute the Dirichlet density of $\mathcal{R}_\gamma(d)$ when $h\ne h(1)$ and $\Delta_K=-3$. We may assume that $h=h(\zeta_3) \ne h(\pm1)$. The case $h=h(-1)$ follows from Theorem \ref{theorem: switch with -1}, which in addition, allows us to switch between $\delta_\gamma$ and $\delta_{-\gamma}$. In particular, when $h=h(\zeta_6)$ or $h=h(\zeta_6^{-1})$, we can switch to $-\gamma$, for which $h=h(\zeta_3)$ or $h=h(\zeta_3^2)$ respectively. The proofs for $h=h(\zeta_3)$ and $h=h(\zeta_3^2)$ are identical, so we present them for $h=h(\omega)$, where $\omega$ is a primitive $3$-rd root of unity. Note that $h\ne h(\pm1)$ implies $3\mid h$.

\begin{theorem}
    Assume $\Delta_K = -3$ and $h =h(\omega)\ne h(\pm1)$. Let $d,k\geq 1$ be integers, $3\nmid d$, and $\mathfrak{f}(L)$ be as defined in Lemma \ref{lemma: cube root in Q(zeta_3,zeta_n)} for $\gamma_0^{h/h_3}$. Then,
    \[
    \delta_\gamma(3^kd) = \delta_{\omega\gamma}(d) \cdot 
    \begin{cases}
        \displaystyle 1-\frac{1}{4\cdot 3^{m}h_3}, & \text{if $k=1$}; \\[0.8em]

        \displaystyle\frac{9}{4\cdot 3^{k+m} h_3}, & \text{if $k\geq 2$},
    \end{cases}
    \]
    where $m=[\mathfrak{f}(L) \mid 9d]$ and $\delta_{\omega\gamma}(d)$ is computed in \cite[Theorem 1.1]{Sa2022} if $2\nmid d$ and in Theorem \ref{theorem: delta(d) with Delta=-3} otherwise.
\end{theorem}

\begin{proof}
    The proof is almost identical to that of Theorem \ref{theorem: Delta=-4 et h=h(i)}, so we may skip a few details. First, we write $\delta_\gamma(3^kd)$ in the same way as \eqref{equation: delta(2^kd) as four sums} with the degree
    \[
    [K_{3^{k+j}dv, 3^{j+luv}}:\Q] = [K_{3^{k+j}dv, uv}:\Q] \cdot [K_{3^{k+j}dv, 3^{j+l}}:K(\zeta_{3^{k+j}dv})].
    \]
    The fields are linearly disjoint over $K(\zeta_{3^{k+j}dv})$ by Kummer theory. Hence the degree formula. The first degree is computed in Lemma \ref{lemma: remaining cases degree}. For the second, note that $\gamma_0^{h/h_3}$ is a cube root in $K(\zeta_{3^{k+j}dv})$ if and only if $\mathfrak{f}(L) \mid 3^{k+j}dv$ by Lemma \ref{lemma: cube root in Q(zeta_3,zeta_n)}. However, because $\mathfrak{f}(L)$ is square-free apart from a factor of $3^n$ ($n\leq 2$), we have  $\mathfrak{f}(L) \mid 3^{k+j}dv$ if and only if  $\mathfrak{f}(L) \mid 9d$ and $j+k\geq 2$. We obtained
    \[
    [K_{3^{k+j}dv, 3^{j+l}}:K(\zeta_{3^{k+j}dv})] = 
    \begin{cases}
        3^l, & \text{if $k=1$ and $j+l\leq v_3(3^mh)$}; \\
        1, & \text{if $k\geq 2$ and $j+l\leq v_3(3^mh)$}; \\
        3^{j+l-v_3(3^mh)}, & \text{if $j+l> v_3(3^mh)$},
    \end{cases}
    \]
    where we used that $K_{3^{k+j}dv, 3^{j+l}}$ is a cyclotomic field when $j+l\leq v_3(3^mh)$, and \cite[Lemma 4]{Perucca2015} and Theorem \ref{theorem: irreducibility of X^n-a} when $j+l>v_3(3^mh)$. The rest of the proof follows as Theorem \ref{theorem: Delta=-4 et h=h(i)}.
\end{proof}

\section{Numerical data}\label{section: numerical data}

In this last section, we present numerical demonstrations of our results. In each table, we write $\gamma$ in the form $u+v\sqrt{\Delta_K}$, where $u,v\in \Q^\times$. For each of them, we computed $h$ and $\zeta$ such that $\zeta\gamma = \gamma_0^{h}$ with $\gamma_0\in K$. We denote by $\Tilde{\gamma}$ the number $\zeta\gamma$. In Tables \ref{table: non-cyclotomic} and \ref{table: Q(i)}, we display the value of $\Tilde{\gamma}^{1/h_2}$ from which one can compute $c$, $K_1$, $K_2$, and their discriminants $\Delta_1$, $\Delta_2$.

In Table \ref{table: non-cyclotomic}, we display $\mathcal{Q}=[N_{K/\Q}(\Tilde{\gamma}^{1/h_2})]$. In Table \ref{table: Q(i)}, we give the value of $\mathfrak{f}(L)$, the conductor of $\Q(i,\Tilde{\gamma}^{1/4h_2})$. Similarly, we write explicitly the value of $\Tilde{\gamma}^{1/h_6}$ in Table \ref{table: Q(w)}, and give the value of $\mathfrak{f}(L)$, the conductor of $\Q(\omega,\Tilde{\gamma}^{1/3h_6})$. Note that we choose $\Tilde{\gamma}^{1/h_6}$ instead of writing $\Tilde{\gamma}^{1/h_2}$ and $\Tilde{\gamma}^{1/h_3}$ separately. This does not change the values $\Delta_1$ and $\Delta_2$, nor of the conductor $\mathfrak{f}(L)$, thanks to B\'ezout's identity.

Finally, for each value of $d$ in the tables, we computed $\delta_\gamma(d)$, which we display as a rational and in numerical form in the ``num.'' column, truncated to $6$ decimal places. In the ``exp.'' column, we provide the experimental value of $\delta_\gamma(d)$ computed via SageMath \cite{sagemath} programs. The programs used in this work are available in the author's GitHub repository \cite{CDC2026sageprograms}. We underline the digits of the ``exp.'' column that coincide with those in the ``num.'' column. We tested primes up to $10^7$.

\renewcommand{\arraystretch}{1.7}
\begin{table}[H]
    \centering
    \begin{tabular}{|c|c|c|c|c|c|c|c|c|}
         \hline
         $\gamma$ & $h$ & $\zeta$ & $\Tilde{\gamma}^{1/h_2}$ & $\mathcal{Q} $& $d$ & $\delta_\gamma(d)$ & num. & exp. \\
         \hline \hline
         $3+\sqrt{8}$ & $2$ & $1$ & $\frac{2+\sqrt{8}}{2}$  & $0$ & $6$ & 17/64 & 0.265625 & \underline{0.2656}70 \\
         \hline
         $ 3+\sqrt{8}$ & $2$ & $1$ & $ \frac{2+\sqrt{8}}{2}$  & $0$ & $20$ & 25/288 & 0.086805 & \underline{0.086}782 \\
         \hline
         $\frac{-27-5\sqrt{29}}{2}$ & $2$ & $-1$ & $\frac{5+\sqrt{29}}{2}$  & $0$ & $8$ & 1/6 & 0.166666 & \underline{0.166}473 \\
         \hline
         $\frac{-27-5\sqrt{29}}{2}$ & $2$ & $-1$ & $\frac{5+\sqrt{29}}{2}$  & $0$ & $10$ & 5/36 & 0.138888 & \underline{0.13}9166 \\
         \hline
         $\frac{17+7\sqrt{-15}}{32}$ & $4$ & $1$ & $\frac{1-\sqrt{-15}}{4}$  & $1$ & $10$ & 5/288 & 0.017361 & \underline{0.017}287 \\
         \hline
         $\frac{17+7\sqrt{-15}}{32}$ & $4$ & $1$ & $\frac{1-\sqrt{-15}}{4}$  & $1$ & $30$ & 5/384 & 0.013020 & \underline{0.0130}17 \\
         \hline
         
    \end{tabular}
    \caption{The case $\Delta_K\not \in \{-4,-3\}$.}
    \label{table: non-cyclotomic}
\end{table}

\begin{table}[H]
    \centering
    \begin{tabular}{|c|c|c|c|c|c|c|c|c|}
         \hline
         $\gamma$ & $h$ & $\zeta$ & $\Tilde{\gamma}^{1/h_2}$ & $\mathfrak{f}(L)$ & $d$ & $\delta_\gamma(d)$ & num. & exp. \\
         \hline \hline
         $\frac{-3+2\sqrt{-4}}{5}$ & $1$ & $1$ & $\frac{-3+2\sqrt{-4}}{5}$ & $20$ & $8$ & 1/3 & 0.333333 & \underline{0.333}427  \\
         \hline
         $\frac{-3+2\sqrt{-4}}{5}$ & $1$ & $1$ & $\frac{-3+2\sqrt{-4}}{5}$ & $20$ & $10$ & 5/72 & 0.069444 & \underline{0.069}279  \\
         \hline
         $\frac{48+7\sqrt{-4}}{50}$ & $2$ & $i$ & $\frac{3+2\sqrt{-4}}{5}$ & $40$ & $10$ & 235/1152 & 0.203993 & \underline{0.203}844  \\
         \hline
         $\frac{48+7\sqrt{-4}}{50}$ & $2$ & $i$ & $\frac{3+2\sqrt{-4}}{5}$ & $40$ & $24$ & 1/16 & 0.062500 & \underline{0.0625}53  \\
         \hline
         $\frac{-240-119\sqrt{-4}}{338}$ & $2$ & $i$ & $\frac{-24+5\sqrt{-4}}{26}$ & $208$ & $26$ & 661/8064 & 0.075768 & \underline{0.0757}71 \\
         \hline
         $\frac{-240-119\sqrt{-4}}{338}$ & $2$ & $i$ & $\frac{-24+5\sqrt{-4}}{26}$ & $208$ & $28$ & 35/288 & 0.121527 & \underline{0.121}457 \\
         \hline
    \end{tabular}
    \caption{The case $\Delta_K=-4$.}
    \label{table: Q(i)}
\end{table}

\renewcommand{\arraystretch}{1.5}
\begin{table}[H]
    \centering
    \begin{tabular}{|c|c|c|c|c|c|c|c|c|}
         \hline
         $\gamma$ & $h$ & $\zeta$ & $\Tilde{\gamma}^{1/h_6}$ & $\mathfrak{f}(L)$ & $d$ & $\delta_\gamma(d)$ & num. & exp. \\
         \hline \hline
         
         $\frac{-13+3\sqrt{-3}}{14}$ & $1$ & $1$  & $\frac{-13+3\sqrt{-3}}{14}$ & 7 & $3$ & 3/4 & 0.750000 & \underline{0.7500}61  \\
         \hline
         $\frac{-13+3\sqrt{-3}}{14}$ & $1$ & $1$  & $\frac{-13+3\sqrt{-3}}{14}$ & 7 & $14$ & 35/288 & 0.121527 & \underline{0.121}231   \\
         \hline

         $\frac{683+37\sqrt{-3}}{686}$ & $3$ & $\omega^2$ & $\frac{1+4\sqrt{-3}}{7}$ & 63 & $9$ & 1/12 & 0.083333 & \underline{0.083}407  \\
         \hline
         $\frac{683+37\sqrt{-3}}{686}$ & $3$ & $\omega^2$ & $\frac{1+4\sqrt{-3}}{7}$ & 63 & $42$ & 1225/10368 & 0.118152 & \underline{0.11}7806 \\
         \hline
         $\frac{1031-520\sqrt{-3}}{1369}$ & $2$ & $-1$  & $\frac{13+20\sqrt{-3}}{37}$ & 333 & $6$ & 5/8 & 0.625000 & \underline{0.62}4809 \\
         \hline
         $\frac{1031-520\sqrt{-3}}{1369}$ & $2$ & $-1$  & $\frac{13+20\sqrt{-3}}{37}$ & 333 & $111$ & 407/16416 & 0.024792 & \underline{0.024}823 \\
         \hline
         
    \end{tabular}
    \caption{The case $\Delta_K=-3$.}
    \label{table: Q(w)}
\end{table}

\section{Acknowledgements}

We are grateful to Carlo Sanna for sharing his computational method for the constant $h$, and to Christian Ballot for his guidance during the author's Ph.D. work. The results in Subsections \ref{subsection: the case Q=0} and \ref{subsection: the case Q=1} were obtained under the latter's supervision.

\bibliographystyle{abbrv}
\bibliography{ref}

\end{document}